\documentclass[11pt]{article}
\usepackage{amsmath,amssymb,cmap,graphicx}

\newcommand{\ee}{{\mathrm e}}
\newcommand{\oo}{{\mathrm o}}

\newcommand{\calV}{{\mathcal V}}
\newcommand{\calW}{{\mathcal W}}
\newcommand{\calX}{{\mathcal X}}

\newcommand{\natu}{{\mathbb N}}
\newcommand{\real}{{\mathbb R}}

\newcommand{\expect}[1]{{\mathbb E}\left\{{\displaystyle #1}\right\}}
\newcommand{\pr}[1]{{\mathbb P}\left\{{\displaystyle #1}\right\}}
\newcommand{\vari}[1]{{\mathbb Var}\left\{{\displaystyle #1}\right\}}
\newcommand{\var}[1]{{\mathbb Var}\left\{{\displaystyle #1}\right\}}

\newtheorem{theorem}{Theorem}

\newtheorem{lemma}[theorem]{Lemma}
\newtheorem{proposition}[theorem]{Proposition}

\newenvironment{proof}[1][Proof]{\begin{trivlist}
\item[\hskip \labelsep {\bfseries #1}]}
{\end{trivlist} \hspace*{\fill} $\Box$}

\def\poiss{\textrm{Po}}

\def\R{\mathbb{R}}
\def\N{\mathbb{N}}

\def\al{\alpha}
\def\bt{\beta}

\def\dl{\delta}
\def\lm{\lambda}
\def\eps{\varepsilon}

\begin{document}

\title{\Large Scaling Laws for Maximum Coloring of Random Geometric Graphs} 

\author{Sem Borst\thanks{Mathematics of Systems, Nokia Bell Labs, 600 Mountain Avenue, Murray Hill NJ 07974, USA, 
and Department of Mathematics and Computer Science, Eindhoven University of Technology, P.O. Box 513, 5600 MB, The Netherlands. 
E-mail: {\tt sem@research.bell-labs.com}.}
\and
Milan Bradonji\'c\thanks{Mathematics of Systems, Nokia Bell Labs, 600 Mountain Avenue, Murray Hill NJ 07974, USA. E-mail: {\tt milan@research.bell-labs.com}.}
}

\date{\today}

\maketitle

\abstract{
We examine maximum vertex coloring of random geometric graphs,
in an arbitrary but fixed dimension, with a constant number of colors.
Since this problem is neither scale-invariant nor smooth, the usual
methodology to obtain limit laws cannot be applied.
We therefore leverage different concepts based on subadditivity
to establish convergence laws for the maximum number of vertices
that can be colored.
For the constants that appear in these results, we provide the exact
value in dimension one, and upper and lower bounds in higher dimensions.
}

\section{Introduction}
\label{sec:intro}

We examine maximum coloring of random geometric graphs (RGGs),
in an arbitrary but fixed dimension~$d$, with a constant number
of colors.
The vertices of an RGG (whose spatial distribution will be defined below)
are embedded in an Euclidean space that is equipped with the $\ell_2$
distance or some $\ell_p$ distance in general, and two vertices are
connected if and only if they are within a given Euclidean distance~$r$.
More specifically, we address the questions:
{\it What is the maximum number of vertices in a sparse RGG that can
be properly colored with a constant number of colors?
In particular, what is the asymptotic behavior of that value,
as the total number of vertices in the graph tends to infinity?}

It is important to emphasize the distinction between our problem and that
of determining the chromatic number, which is the minimum required
number of colors to properly color all the vertices of a graph such that
no two adjacent vertices are assigned the same color.
Determining whether an RGG (or a unit-disk graph) is $k$-colorable,
i.e., whether its chromatic number is at most~$k$, is NP-hard even
for $k=3$, see~\cite{brent-1990-unitdiskgraphs}.
Our problem is different from determining the chromatic number,
since we are interested in the maximum number of vertices that can
be properly colored with given $k \in \N$ colors,
as well as from $k$-colorability, which is a binary decision problem.
The chromatic number of RGGs has been studied in detail
(for different values of the expected degree),
see Theorem 1.1 in~\cite{mcdiarmid-2011-chromatic}.
The chromatic number in the {\em thermodynamic regime}, when the expected degree is constant, is `almost'
logarithmic in the number of vertices~$n$, i.e., $(1+o(1))\log n / \log\log n$,
which additionally inspires our problem where only a constant number
of colors is available.

The above-mentioned questions are not only of fundamental interest,
but also motivated by applications in wireless networks, where the various
users need to be assigned channels (transmission frequencies) in order
to be able to communicate, subject to certain interference constraints.
For example, in order to avoid excessive interference, the same channel
cannot be assigned to two users within a certain reuse distance~$r$.
The total number of required channels to cover all users then
corresponds to the chromatic number of the associated interference graph
where two users are neighbors when they are located within distance~$r$.
When the user locations are governed by a spatial Poisson process,
the interference graph is an RGG, and the chromatic number will grow
without bound as the total number of users grows large.
As a result, the required number of channels to cover all users will
grow without bound, implying that the capacity per channel,
and hence the so-called max-min throughput of the network, will vanish
in the limit, which is obviously undesirable.
The question thus arises how many users can be covered when
the number of available channels is finite.
It will then not be feasible to cover all users as the total number
of users grows large, but the users that do get covered are ensured
to receive a strictly positive throughput.
The results that we prove in the present paper imply that any target
for the fraction of users to be covered, arbitrarily close to one,
can be achieved in the limit with a sufficiently large but constant
number of channels.

Besides wireless networks, RGGs have also found applications in
various further areas, e.g.~cluster analysis, statistical physics,
modeling data in high-dimensional spaces, and hypothesis testing,
to mention just a few~\cite{penrose:book}.
For problems on many of these `real' networks, the sparse regime
with constant expected vertex degree is particularly relevant,
see~\cite{DBLP:journals/im/LeskovecLDM09}.

We now formally state the main problem and results.
For any subset of points $V \subseteq \real^d$ and $r \in \real_+$,
let $G_r(V)$ be the graph with vertex set~$V$ and edge set
$E = \{\{u, v\} \in V^2: ||u - v|| \leq r\}$, i.e., connecting all
pairs of points that are within a given Euclidean distance~$r$.

The main object of interest is the cardinality of a set obtained
by a maximum proper coloring with $k$ colors of a given graph $G_r(V)$.
Note that such a set obtained by a maximum proper coloring on finite $V$ 
may not be unique, but its cardinality is unique and defined as follows.
For any $k \in \natu$, let $N_{k,r}(V)$ be the maximum number
of vertices that can be properly colored in $G_r(V)$ with $k$~colors.
For any $\lambda > 0$, let $\calX_\lambda$ be a Poisson point process
of intensity~$\lambda$ in $\real^d$. For compactness, denote
\[
F_{k,\lambda}(t) = N_{k,1}([0, t]^d \cap \calX_\lambda)
\]
for any $t \geq 0$.

Also, let ${\mathcal I}_n$ be a collection of $n$~points uniformly
and independently distributed in the unit cube $[0, 1]^d$.
For compactness, denote
\[
H_{k,r}(n) = N_{k,r}({\mathcal I}_n)
\]
for any $n \in \natu$ and $r > 0$.

\textbf{The main problem:} We are interested in the asymptotic behavior
of the expectation and moreover the distribution of $F_{k,\lambda}(t)$
as $t \to \infty$, as well as $H_{k,\nu}(n)$ as $n \to \infty$.

\textbf{The main results:} 
We show that for any $d, k \in \N$ and $\lambda > 0$, the functional
$F_{k,\lambda}(t)$ converges in probability
\[
\frac{F_{k,\lambda}(t)}{\lambda t^d} \stackrel{\textrm{p}}{\to} a_{k,\lambda} \,,
\]
for some $a_{k,\lambda} \in (0,1]$, and in distribution, for any $\nu > 0$,
\[
\frac{H_{k,\sqrt[d]{\nu / n}}(n)}{n} \stackrel{\textrm{d}}{\to} a_{k,\nu} \,.
\]

One of our main methods involves the notion of {\it subadditivity}.
Concretely, we divide the cube $[0,t]^d$ into cubes of volume $s^d$,
for some $s<t$ which we specify later, and apply the subadditivity
argument in order to relate $F_{k,\lm}(t)$ and $F_{k,\lm}(s)$.
We show that the lower and upper limits as $t \to \infty$
of $F_{k,\lm}(t)$ exist and are the same, and moreover we establish
the weak law of large numbers for $F_{k,\lm}(t)$ 
and the strong law of large numbers for $H_{k,\nu}(n)$.

In Lemma~\ref{lemma:sigma.lb}, we prove that the variance
$\vari{F_{k,\lm}(t)} = \Omega(t^d)$, i.e.,~the limiting variance 
normalized by $t^d$ is bounded away from~$0$,
and in Lemma~\ref{lemma:sigma.ub.1} we present an upper bound
on $\vari{F_{k,\lm}(t)} =O(t^d)$, which together imply
$\vari{F_{k,\lm}(t)} = \Theta(t^d)$, see Lemma~\ref{lemma:sigma.lub}. 

There are two branches of methods prevalent in discrete stochastic
geometry, subadditive and stabilization methods, usually used to obtain
the limiting behavior of some Euclidean functionals: 
laws of large numbers, central limit theorems, etc.
For an excellent survey, the reader is referred to Yukich~\cite{yukich-2013-limit}.

At first glance, the results in this paper can be seen as a subproblem
and amenable to analysis by using techniques from~\cite{bhh-1959},
and even the ``more general subadditive methods'' developed by Steele
in Chapter~$3$ of~\cite{steele-1997-probability}.
In order to apply these techniques from~\cite{steele-1997-probability},
a function $L$ that maps a finite subset of points from $\real^d$
to $\real_+$ must satisfy the following four hypotheses:
(i) normalization $L(\emptyset) = 0$;
(ii) homogeneity $L\left(\al x_1, \al x_2,\dots,\al x_n\right) =
\al L\left(x_1,x_2,\dots,x_n\right)$ for every $\al>0$;
(iii) translation invariance 
$\forall y \in \R^d$
$L\left(x_1+y,x_2+y,\dots,x_n+y\right) = 
L\left(x_1,x_2,\dots,x_n\right)$;
(iv) geometric subadditivity, where for all $m, n \geq 1$
and $x_1,x_2,\dots,x_n \in [0,1]^d$ we have
\begin{equation}
L\left(x_1,x_2,\dots,x_n\right) \leq \sum_{i=1}^{m^d}
L\left(\{x_1,x_2,\dots,x_n\} \cap Q_i\right) + O(m^{d-1}) \,,
\end{equation}
where the unit cube $[0,1]^d$ is partitioned into $m^d$ cubes $Q_i$
with side $1/m$.

Additionally, (v) $L$ is monotone, if for all~$n$ and $x_i$,
$L\left(x_1,\dots,x_n\right) \leq L\left(x_1,\dots,x_n,x_{n+1}\right)$.
For example, Steele proves a so-called ``basic theorem'',
see Theorem 3.1.1 in~\cite{steele-1997-probability},
for general subadditive
Euclidean functionals that are monotone and satisfy the four
conditions (i)--(iv).
This theorem states that if $x_1, x_2, \dots, x_n$ are independent
random variables uniformly distributed on $[0,1]^d$ then with probability one
\begin{equation}
\label{eq:L.limit}
\lim_ {n \to \infty} n^{-(d-1)/d} L\left(x_1,x_2,\dots,x_n\right) = \beta_L(d) \,,
\end{equation}
where $\bt_L(d)$ is a positive constant, which depends both on the
functional $L$ and the dimension~$d$.

In order to show applications of Theorem 3.1.1 from~\cite{steele-1997-probability}, 
we first mention, one of the classical problems in combinatorial optimization, the famous 
traveling salesman problem (TSP), where the goal is to find the minimum-length Hamiltonian tour 
among a list of cities, given the distances between each pair of cities.
For more details on TSP, the reader is referred to~\cite{applegate:tsp:book}.

One subcase of TSP is Euclidean TSP, where the objective is to determine the minimum-length Hamiltonian tour
of $n$~points $x_1,x_2,\dots,x_n$ distributed in $\R^d$, and analyze its limiting behavior as $n$ grows.
Euclidean TSP was originally studied and the limiting behavior (\ref{eq:L.limit}) was shown by Beardwood, Halton and Hammersley in~\cite{bhh-1959}. 

Euclidean TSP is one representative example to apply Theorem 3.1.1 from~\cite{steele-1997-probability}.
It is clear that the length of the minimum tour satisfies (i)--(v). 
In other words, the minimum-length Hamiltonian tour on $n$ points: 
(i) is equal to $0$ if the number of points $n=0$ (normalization);
(ii) increases $\al$ times, for every $\al>0$, if the positions
of the points are rescaled $\al$ times (homogeneity);
(iii) is translation invariant (does not change under the translation
of the points);
(iv) is subadditive, see~\cite{steele-1997-probability,goemans-1991-prob};
and 
(v) is monotone by applying the triangle inequality on three distances among 
$x_{n+1}$ and its two adjacent points in the optimal tour on $x_1,x_2,\dots,x_{n+1}$. 
Besides Euclidean TSP, subadditive Euclidean functionals, such as, the length of the least Euclidean matching, Steiner tree, and rectilinear Steiner tree, satisfy the conditions of Theorem 3.1.1 from~\cite{steele-1997-probability}, 
and their asymptotic growth rates have been studied in~\cite{subadditive-1981-steele}.

However, the functionals of interest in this work $F_{k,\lm}(t)$
and $H_{k,\nu}(n)$ satisfy normalization, translation invariance,
and subadditivity, but {\em not\/} the homogeneity property, i.e.,
$N_{k,r}(\al V) \neq \al N_{k,r}(V)$, where $\al > 0$ and $V \subset \R^d$.

Moreover, we later show that 
$N_{k,r}(V)$ is {\em not continuous\/}, also called {\em not smooth}.
In more detail, for any $V$ and $V'$ at distance at least~$r$, 
$|N_{k,r}(V \cup V') - N_{k,r}(V)| = |N_{k,r}(V')| = \Theta(V')$,
which is not $O(|V'|^{(d-p)/d})$ for some $p>0$, hence one cannot apply the methods from~\cite{rhee-1993-matching}.
For example, TSP and MST (Minimum Spanning Tree) are smooth
functionals.
In~\cite{yukich-2013-limit}, Yukich presents laws of large numbers
for smooth, superadditive Euclidean functionals, see Theorem 8.1,
as well as the general `umbrella theorem' for subadditive, smooth
Euclidean functionals, see Theorem 8.3.

Additionally, there is a large body of work, which uses the
{\it stabilization\/} method, see Penrose~\cite{penrose:book},
Yukich~\cite{yukich-2013-limit}.
It is not hard to show that stabilization holds in the sub-critical
regime for $\lm < \lm_c$, when a.a.s~\footnote{a.a.s. stands for ``asymptotically almost surely''.}
no infinite cluster exists,
but it is not clear at all whether stabilization holds in the
super-critical regime $\lm > \lm_c$ that we are interested in.
Moreover, the stabilization is a very `effective' method, see Penrose~\cite{penrose:book}, which needs
assumptions on a considered functional
that seem rather strong and that we do not impose in our proofs.

\section{Main Properties and Preliminary Results}
\label{sec:def}

As introduced before, for any subset of points $V \subseteq \real^d$
and $r \in \real_+$, let $G_r(V)$ be the graph with vertex set~$V$
and edge set $E = \{\{u, v\} \in V^2: ||u - v|| \leq r\}$, i.e.,
connecting all pairs of points that are within a given distance~$r$.
For any $k \in \natu$, let $N_{k,r}(V)$ be the maximum number of
vertices that can be properly colored in $G_r(V)$ with $k$~colors.

We state a few basic but useful properties.
\newline
(0)
$N_{k,r}(\alpha V) = N_{k, r / \alpha}(V)$, for all $\al>0$,
since the edge sets of $G_r(\alpha V)$ and $G_{r / \alpha}(V)$ coincide.
\newline
(i) Inhomogeneity ({\em no\/} scale-invariance):
There exists $\al>0$ such that $N_{k,r}(\alpha V) \neq \alpha N_{k, r}(V)$,
and in general the inequality holds for many values $\al>0$.''
\newline
(ii) Monotonicity:
If $U \subseteq V$, then $N_{k,r}(U) \leq N_{k,r}(V)$.
Also, $N_{k,r}(V)$ is decreasing in~$r$.
\newline
(iii) Subadditivity:
$N_{k,r}(V_1 \cup V_2) \leq N_{k,r}(V_1) + N_{k,r}(V_2)$ for any
$V_1$, $V_2$.
\newline
(iv) Near-superadditivity:
If $V_1, V_2 \subseteq V$ are at a distance more than~$r$, i.e.,
$||v_1 - v_2|| > r$ for all $v_1 \in V_1$, $v_2 \in V_2$,
then $N_{k,r}(V_1 \cup V_2) \geq N_{k,r}(V_1) + N_{k,r}(V_2)$. 

By virtue of the property $(0)$ and the properties of
the spatial Poisson process,
\[
N_{k,r}([0, t]^d \cap \calX_\lambda) \stackrel{{\mathrm st}}{=}
N_{k,1}(([0, t]^d \cap \calX_\lambda)/r) \stackrel{{\mathrm st}}{=}
N_{k,1}(([0, t/r]^d \cap \calX_{\lambda r^d}) = F_{k,\lambda r^d}(t/r),
\]
and hence we may focus the attention on the random variable
$F_{k,\lambda}(t)$ without loss of generality.

Also, it follows from the above-mentioned monotonicity property
and a simple coupling argument that $H_{k,\nu}(n)$ is
stochastically increasing in~$n$. 

For any $\mu \in \real_+$, let $\poiss(\mu)$ be a Poisson random
variable with parameter~$\mu$ and $\pi(\mu, m) = \pr{\poiss(\mu) = m}$.
The property $(0)$ and the properties of the spatial 
Poisson process furnish a useful relationship between the variables
$F_{k,\lambda}(t)$ and $H_{k,\nu}(n)$:
\[
F_{k, \lambda}(t) 
\stackrel{{\mathrm st}}{=} N_{k,1/t}(([0, t]^d \cap \calX_\lambda)/t)
\stackrel{{\mathrm st}}{=} N_{k,1/t}([0, 1]^d \cap \calX_{\lambda t^d})
\stackrel{{\mathrm st}}{=} H_{k,1/t}(\poiss(\lambda t^d)),
\]
and in particular,
\begin{equation}
\expect{F_{k,\lambda}(t)} = \expect{H_{k,1/t}(\poiss(\lambda t^d))} =
\sum_{l = 0}^{\infty} \expect{H_{k,1/t}(l)} \pi(\lambda t^d, l).
\label{rela1}
\end{equation}

For any $s, t \in \real_+$, define
\[
M^u(s, t) = \left\lceil\frac{t}{s}\right\rceil^d \,,
\]
and
\[
M^l(s, t) = \left\lfloor\frac{t}{s + 1}\right\rfloor^d \,.
\]

Invoking the stationarity (translation invariance) of the spatial
Poisson process, the subadditivity property implies
\begin{equation}
F_{k,\lambda}(t) \leq_{{\mathrm st}} \sum_{i = 1}^{M^u(s, t)} N_i(s)\,,
\label{ub1}
\end{equation}
while the near-superadditivity property implies
\begin{equation}
F_{k,\lambda}(t) \geq_{{\mathrm st}} \sum_{i = 1}^{M^l(s, t)} N_i(s)\,,
\label{lb1}
\end{equation}
where $N_1(s), N_2(s), \dots$ are i.i.d.\ copies of the random
variable $F_{k,\lambda}(s)$.

For compactness, denote
$\overline{F}_{k,\lambda}(t) = \expect{\widetilde{F}_{k,\lambda}(t)}$, with
\[
\widetilde{F}_{k,\lambda}(t) = \frac{F_{k,\lambda}(t)}{\lambda t^d},
\]
as the `coloring ratio', i.e., the ratio of the expected maximum number
of vertices that can be colored and the expected total number of vertices.

The stochastic upper and lower bounds~(\ref{ub1}) and~(\ref{lb1}) yield
\begin{equation}
\widetilde{F}_{k,\lambda}(t) \leq_{{\mathrm st}}
\left(\frac{s}{t}\right)^d \sum_{i = 1}^{M^u(s, t)} \widetilde{N}_i(s)\,,
\label{ub2}
\end{equation}
and
\begin{equation}
\widetilde{F}_{k,\lambda}(t) \geq_{{\mathrm st}}
\left(\frac{s}{t}\right)^d \sum_{i = 1}^{M^l(s, t)} \widetilde{N}_i(s)\,,
\label{lb2}
\end{equation}
where $\widetilde{N}_1(s), \widetilde{N}_2(s), \dots$ are i.i.d.\ copies
of the random variable $\widetilde{F}_{k,\lambda}(s)$.

\begin{proposition}
\label{prop:one}
For any $k \geq 1$, $\lambda > 0$, the limit
$\lim_{t \to \infty} \overline{F}_{k,\lambda}(t)$ exists, and equals
$a_{k,\lambda} := \inf_{t > 0} \overline{F}_{k,\lambda}(t)$.
\end{proposition}

\begin{proof}
Taking expectations in the stochastic upper bound~(\ref{ub2}), we obtain
\[
\overline{F}_{k,\lambda}(t) \leq
\left(\frac{s}{t}\right)^d M^u(s, t) \overline{F}_{k,\lambda}(s) \,.
\]

Note that for any fixed~$s$, 
\[
\lim_{t \to \infty} \left(\frac{s}{t}\right)^d M^u(s, t) = 1 \,,
\]
i.e., for any $\delta > 0$, there exists $t(s, \delta)$ such that
\[
\left(\frac{s}{t}\right)^d M^u(s, t) \leq 1 + \delta
\]
for all $t \geq t(s, \delta)$.

This yields
\[
\limsup_{t \to \infty} \overline{F}_{k,\lambda}(t) \leq
(1 + \delta) \overline{F}_{k,\lambda}(s)
\]
for any fixed~$s$, hence
\[
\limsup_{t \to \infty} \overline{F}_{k,\lambda}(t) \leq (1 + \delta)
\inf_{s > 0} \overline{F}_{k,\lambda}(s) \,,
\]
and therefore
\[
\limsup_{t \to \infty} \overline{F}_{k,\lambda}(t) \leq
\inf_{s > 0} \overline{F}_{k,\lambda}(s) \,,
\]
since $\delta > 0$ is arbitrary.

Because obviously
\[
\liminf_{t \to \infty} \overline{F}_{k,\lambda}(t) \geq
\inf_{s > 0} \overline{F}_{k,\lambda}(s) \,,
\]
we deduce
\[
\lim_{t \to \infty} \overline{F}_{k,\lambda}(t) =
\inf_{s > 0} \overline{F}_{k,\lambda}(s) \,.
\]
\end{proof}

The next proposition provides two useful properties for the constant
$a_{k,\lambda}$ introduced in Proposition~\ref{prop:one}.

\begin{proposition}

For any $\epsilon > 0$,
\[
a_{k,\lambda (1 + \epsilon)} \leq a_{k,\lambda} + \epsilon \,,
\]
and thus
\[
a_{k,\lambda (1 - \epsilon)} \geq a_{k,\lambda} - \frac{\epsilon}{1 - \epsilon} \,.
\]

\end{proposition}

\begin{proof}
By virtue of the subadditivity property and the properties of the
spatial Poisson process, for any $\epsilon > 0$, $t > 0$,
\begin{eqnarray*}
F_{k,\lambda (1 + \epsilon)}(t)
&=&
N_{k,1}([0, t]^d \cap \calX_{\lambda (1 + \epsilon)}) \\
&\stackrel{{\mathrm st}}{=}&
N_{k,1}(([0, t]^d \cap \calX_\lambda) \cup
([0, t]^d \cap \calX'_{\epsilon \lambda})) \\
&\leq_{{\mathrm st}}&
N_{k,1}([0, t]^d \cap \calX_\lambda) +
N_{k,1}([0, t]^d \cap \calX_{\epsilon \lambda}) \\
&\leq&
N_{k,1}([0, t]^d \cap \calX_\lambda) + |[0, t]^d \cap \calX_{\epsilon \lambda}| \,.
\end{eqnarray*}

Taking expectations and dividing by $\lambda t^d$ yields
\[
\overline{F}_{k,\lambda (1 + \epsilon)}(t) \leq
\overline{F}_{k,\lambda}(t) + \epsilon \,.
\]

Taking the limit over $t > 0$, we obtain the statement of the
proposition.
\end{proof}

For any $k \geq 1$, $\nu > 0$, define
$\widetilde{H}_{k, \nu}(n) = H_{k, \nu}(n) / n$.

\begin{proposition}
\label{prop:fkl.convergence}
For any $k \geq 1$, $\nu > 0$,
\[
\lim_{n \to \infty} \expect{\widetilde{H}_{k,\sqrt[d]{\nu / n}}(n)} =
a_{k,\nu} \,.
\]
\end{proposition}

\begin{proof} 
The proof relies on the relationship~(\ref{rela1}) and the
monotonicity property of $H_{k,\nu}(n)$, in conjunction with
asymptotic lower and upper bounds.

We first consider the upper bound.

Using~(\ref{rela1}) and observing that $\expect{H_{k,1/t}(l)}$ is
increasing in~$l$, we obtain
\[
\expect{F_{k,\lambda}(t)} \geq
\expect{H_{k,1/t}(n)} \pr{\poiss(\lambda t^d) \geq n}\,,
\]
for any $n \in \natu$.

Taking $t = \sqrt[d]{n / \nu}$ and $\lambda = \nu (1 + \epsilon)$,
we deduce
\[
\expect{F_{k, \nu (1 + \epsilon)}(\sqrt[d]{n / \nu})} \geq
\expect{H_{k,\sqrt[d]{\nu / n}}(n)} \pr{\poiss(n (1 + \epsilon)) \geq n}\,.
\]

Noting that $\pr{\poiss(n (1 + \epsilon)) \geq n} \to 1$
as $n \to \infty$, Proposition~\ref{prop:one} then implies
\[
\limsup_{n \to \infty} \frac{\expect{H_{k,\sqrt[d]{\nu / n}}(n)}}{n} \leq (1 + \epsilon) \limsup_{n \to \infty} \frac{\expect{F_{k, \nu (1 + \epsilon)}(\sqrt[d]{n / \nu})}}{\nu (1+\epsilon) n / \nu} = (1 + \epsilon) a_{k, \nu (1 + \epsilon)}.
\]

Since $\epsilon > 0$ is arbitrary, it then follows from Proposition~2 that
\[
\limsup_{n \to \infty} \frac{\expect{H_{k,\sqrt[d]{\nu / n}}(n)}}{n} \leq
a_{k, \nu} \,.
\]

We now turn to the lower bound.

Invoking~(\ref{rela1}) and noting that $\expect{H_{k,1/t}(l)}$ is
increasing in~$l$, we obtain
\[
\expect{F_{k,\lambda}(t)} \leq \expect{H_{k,1/t}(n)} +
\sum_{l = n}^{\infty} l \pr{\poiss(\lambda t^d) = l},
\]
for any $n \in \natu$.

Choosing $t = \sqrt[d]{n / \nu}$ and $\lambda = \nu (1 - \epsilon)$,
we deduce
\[
\expect{F_{k,\nu (1 - \epsilon)}(\sqrt[d]{n / \nu})} \leq
\expect{H_{k,\sqrt[d]{\nu / n}}(n)} +
\sum_{l = n}^{\infty} l \pr{\poiss(n (1 - \epsilon)) = l}.
\]

Invoking that
$\sum_{l = n}^{\infty} l \pr{\poiss(n (1 - \epsilon)) = l} = \oo(1)$
as $n \to \infty$, it then follows from Proposition~\ref{prop:one} that
\[
\liminf_{n \to \infty} \frac{\expect{H_{k,\sqrt[d]{\nu / n}}(n)}}{n} \geq (1 - \epsilon) \liminf_{n \to \infty} \frac{\expect{F_{k, \nu (1 - \epsilon)}(\sqrt[d]{n / \nu})}}{\nu (1-\epsilon) n / \nu} = (1 - \epsilon) a_{k, \nu (1 - \epsilon)}.
\]

Since $\epsilon > 0$ is arbitrary, it then follows from Proposition~2 that
\[
\liminf_{n \to \infty} \frac{\expect{H_{k,\sqrt[d]{\nu / n}}(n)}}{n} \geq
a_{k, \nu}.
\]
\end{proof}

For conciseness, denote $\mu(t) = \expect{F_{k, \lambda}(t)}$
and $\sigma^2(t) = \vari{F_{k, \lambda}(t)}$.

\begin{lemma}
\label{lemma:sigma.lb}
For all $\lambda, d, k$,
\[
\inf_{t > 3} \frac{\sigma^2(t)}{t^d} > 0 \,.
\]
\end{lemma}

\begin{proof}
Let $C^d(t) = \left(\lfloor\frac{t}{3}\rfloor\right)^d$ and let
$\calV = v_1, v_2, \dots, v_{C^d(t)}$ be a set of points in $[0, t]^d$
such that no two points are within distance~3 and no point is within
distance $3 / 2$ from the boundary.
For each point~$v_i$ consider two spheres with radii $1 / 2$ and $3 / 2$
centered on~$v_i$, along with the shell of unit width consisting of the
area covered by the larger sphere but not by the smaller sphere.
Let $\calW \subseteq \calV$ be the set of points for which the associated
shells do not contain any points of the Poisson point process $\calX_\lambda$.
Note that
\[
\pr{\calW} = q^{|\calW|} (1 - q)^{C^d(t) - |\calW|}\,,
\]
with
\[
q = \ee^{- \lambda (V^d(3 / 2) - V^d(1 / 2))} > 0\,,
\]
and
\[
V^d(r) = \frac{\pi^{d / 2} r^d}{\Gamma(d / 2 + 1)}
\]
denoting the volume of a sphere of radius~$r$ in $d$~dimensions.
Given a maximum proper coloring with $k$ colors of $[0, t]^d \cap \calX_\lambda$,
let $D_i \in \{0, 1, \dots, k\}$ be the number of colored points covered
by the smaller sphere around~$v_i$
and $D_0 = F_{k, \lambda}(t) - \sum_{i \in \calW} D_i$.
We may write
\[
\sigma^2(t) = \expect{(F_{k, \lambda}(t) - \mu(t))^2} =
\sum_{\calW \subseteq \{1, \dots, C^d(t)\}}
\expect{(F_{k, \lambda}(t) - \mu(t))^2 | \calW} \pr{\calW}.
\]
Now observe that
\[
D_i \stackrel{{\mathrm st}}{=} \min\{k, \poiss(\lambda V^d(1 / 2))\},
\]
for all $i \in \calW$, and that $D_0$ and $D_i$, $i \in \calW$,
are all mutually independent conditioned on $\calW$.
Denoting $\mu_k = \expect{\min\{k, \poiss(\lambda V^d(1 / 2))\}}$,
and $\sigma_k^2 = \vari{\min\{k, \poiss(\lambda V^d(1 / 2))\}}$, we deduce
\begin{eqnarray*}
\expect{\left(F_{k, \lambda}(t) - \mu(t) \right)^2 | \calW}
&=&
\expect{\left(D_0 + \sum_{i \in \calW} D_i - \mu(t)\right)^2 | \calW} \\
&=&
\expect{\left(D_0 - (\mu(t) - |\calW| \mu_k) +
\sum_{i \in \calW} (D_i - \mu_k)\right)^2 | \calW} \\
&=&
\expect{\left(D_0 - \left(\mu(t) - |\calW| \mu_k \right) \right)^2 | \calW} +
\sum_{i \in \calW} \expect{\left(D_i - \mu_k \right)^2 | \calW} \\
&\geq&
|\calW| \sigma_k^2 \,.
\end{eqnarray*}
Combining the above properties yields
\[
\sigma^2(t) \geq 
\sum_{\calW \subseteq \{1, \dots, C^d(t)\}} |\calW| \sigma_k^2
q^{|\calW|} (1 - q)^{C^d(t) - |\calW|} =
\sigma_k^2 \sum_{m = 0}^{C^d(t)}
\left(\begin{array}{c} m \\ C^d(t)\end{array}\right) m q^m (1 - q)^{C^d(t) - m} =
q C^d(t) \sigma_k^2 \,.
\]
Since $q > 0$, $\sigma_k^2 > 0$, regardless of~$t$, and
$C^d(t) \geq (t / 6)^d$ for all $t > 3$, the statement of the lemma follows
with 
\begin{equation}
\inf_{t > 3} \frac{\sigma^2(t)}{t^d} \geq q \sigma_k^2 / 6^d > 0\,.
\end{equation}
\end{proof}

\begin{lemma}
\label{lemma:sigma.H.ub}
For any $n \in \natu$,
\[
\var{H_{k, \sqrt[d]{\nu / n}}(n)} \leq n/2 \,.
\]
\end{lemma}

\begin{proof}

Recall that ${\mathcal I}_n$ is a collection of $n$~points uniformly
and independently distributed in the unit cube $[0, 1]^d$.
Let ${\mathcal I}_n^i$ be a collection of $n$~points obtained
from ${\mathcal I}_n$ by replacing (only) the $i$-th point $x_i$
by $x'_{i} \in [0,1]^d$.
By definition, $H_{k,r}(n)=N_{k,r}({\mathcal I}_n)$. 
Moreover, $N_{k,r}({\mathcal I}_n)$ is one-Lipschitz
\begin{equation}
\label{eq:N_one_Lipschitz}
\left| N_{k,r}({\mathcal I}_n) - N_{k,r}({\mathcal I}^{i}_n) \right| \leq 1 \,, 
\end{equation}
by the nature of a maximum proper coloring with $k$ colors.
Now the one-Lipschitz property of $N_{k,r}({\mathcal I}_n)$ and the
Efron-Stein bound~\cite{efron-1981-jackknife} yield the assertion
of the lemma
\begin{eqnarray*}
\var{H_{k, \sqrt[d]{\nu / n}}(n)} &=&  \var{N_{k, \sqrt[d]{\nu / n}}({\mathcal I}_n)} \\
&\leq& \frac{1}{2} \expect{ \sum_{i=1}^n \left(N_{k, \sqrt[d]{\nu / n}} \left({\mathcal I}_n\right) - N_{k, \sqrt[d]{\nu / n}}\left({\mathcal I}^{i}_n\right) \right)^2} \\ 
&\leq& n/2 \,.
\end{eqnarray*}

\end{proof}

\begin{lemma}
\label{lemma:sigma.ub.1}
For any $\lambda, d, k$ and for every $\eps>0$ and sufficiently large $t$,
\[
\sigma^2(t) \leq (1+\eps) \lm t^d /2 \,.
\]
\end{lemma}

\begin{proof}
Conditioning on the number of points in $\calX_{\lm} \cap [0,t]^d$,
which is a Poisson random variable and concentrates around its mean
$\lm t^d$, see~\cite{penrose:book}, the proof follows by using the
bound on the variance of $H_{k, \sqrt[d]{\nu / n}}(n)$ provided
by Lemma~\ref{lemma:sigma.H.ub}.
\end{proof}

\begin{lemma}
\label{lemma:sigma.lub}
For all $\lambda>0$ and $d, k \in \N$, asymptotically as $t$ tends to $\infty$, we have 
\begin{equation}
\nonumber
\sigma^2(t) = \Theta(t^d)\,.
\end{equation}
\end{lemma}

\begin{proof}
The proof follows from Lemmas~\ref{lemma:sigma.lb} and~\ref{lemma:sigma.ub.1}.
\end{proof}

\section{Convergence of functionals $\widetilde{F}_{k, \lambda}(t)$
and $\widetilde{H}_{k,\sqrt[d]{\nu / n}}(n)$}
\label{sec:lln}

The next two propositions establish convergence in probability,
i.e., a weak law-of-large-numbers result for the two functionals
$\widetilde{F}_{k, \lambda}(t)$ and $\widetilde{H}_{k,\sqrt[d]{\nu / n}}(n)$,
and a strong law-of-large-numbers result (almost-sure convergence)
of $\widetilde{H}_{k,\sqrt[d]{\nu / n}}(n)$.

\begin{proposition}
\label{prop:H.conv.in.prob}
For any $k \geq 1$, $\nu > 0$, the random variable
$\widetilde{H}_{k,\sqrt[d]{\nu / n}}(n)$ almost surely converges
to $a_{k,\nu}$ as $n \to \infty$.
\end{proposition}

The first part of the proof below establishes the convergence
in probability, while the second part concludes with the almost-sure
convergence. 

\begin{proof}
We first show the convergence in probability.
We have that $N_{k, r}(V)$ is one-Lipschitz, see~\ref{eq:N_one_Lipschitz}.\
Hence a `simple concentration bound', see p.79 (Section 10.1)
of~\cite{molloy-reed-book-2002}, implies that for any $\eps \geq 0$
\[
\pr{\left|H_{k, \sqrt[d]{\nu / n}}(n) - 
\expect{H_{k, \sqrt[d]{\nu / n}}(n)}\right| \geq \eps} \leq 2 \ee^{- \eps^2 /2n} \,.
\] 
Moreover, for any $\delta > 0$, Proposition~\ref{prop:fkl.convergence}
implies that there exists an $n = n_\delta$ such that
\[
\left|\expect{\widetilde{H}_{k, \sqrt[d]{\nu / n}}(n)} - a_{k, \lambda} \right| \leq
\frac{1}{2} \delta
\]
for all $n \geq n_\delta$.
Thus, for all $n \geq n_\delta$,
\begin{eqnarray}
\label{eq:exp.bound.conv.prob}
\nonumber
\pr{\left|\widetilde{H}_{k,\sqrt[d]{\nu / n}}(n) - a_{k,\nu}\right| > \delta}
&\leq&
\pr{\left|\widetilde{H}_{k,\sqrt[d]{\nu / n}}(n) - 
\expect{\widetilde{H}_{k, \sqrt[d]{\nu / n}}(n)}\right| > \frac{1}{2} \delta} \\
\nonumber
&=&
\pr{\left|H_{k,\sqrt[d]{\nu / n}}(n) - 
\expect{H_{k, \sqrt[d]{\nu / n}}(n)}\right| > \frac{n}{2} \delta} \\
&\leq&
2 \ee^{- n \delta^2 / 2} \,.
\end{eqnarray}
For any $\delta > 0$, the latter term tends to~$0$ as $n$ tends to $\infty$,  
which implies the convergence in probability of $\widetilde{H}_{k,\sqrt[d]{\nu / n}}(n)$.

We now proceed to establish the almost-sure convergence.
From~(\ref{eq:exp.bound.conv.prob}), for any $\delta > 0$ and the corresponding finite $n_\delta$, it follows:
\begin{equation}
\sum_{n=1}^\infty
\pr{\left|\widetilde{H}_{k,\sqrt[d]{\nu / n}}(n) - a_{k,\nu}\right| > \delta}
\leq
n_\dl + 2 \sum_{n=n_\dl}^\infty \ee^{- n \delta^2 / 2} < \infty \,,
\end{equation}
which implies the almost-sure convergence of $\widetilde{H}_{k,\sqrt[d]{\nu / n}}(n)$.
\end{proof}

Given Proposition~\ref{prop:H.conv.in.prob}, we prove the convergence
in probability of $\widetilde{F}_{k,\lambda}(t)$.

\begin{proposition}
\label{prop:lln1}
For any $k \geq 1$, $\lambda > 0$, the random variable
$\widetilde{F}_{k,\lambda}(t)$ converges to $a_{k, \lambda}$
in probability as $t \to \infty$.
\end{proposition}

\begin{proof}
The proof relies on lower and upper bounds, which asymptotically
coincide.

For both bounds, we will use the fact that
\[
\vari{\widetilde{F}_{k, \lambda}(s)} \leq \frac{1}{\lambda s^d}
\]
because of Lemma~\ref{lemma:sigma.ub.1}, and hence
\begin{equation}
\pr{\left|\frac{1}{M} \sum_{i = 1}^{M} \widetilde{N}_i(s) -
\expect{\widetilde{F}_{k, \lambda}(s)}\right| \geq \eps} \leq
\frac{1}{M \lambda s^d \eps^2}
\label{chebyshev1}
\end{equation}
for any $M \geq 1$ and $\eps > 0$ by virtue of Chebyshev's inequality.

We first consider the lower bound.

Proposition~\ref{prop:one} implies that there exists an $s = s_\delta$ such that
\[
\overline{F}_{k, \lambda}(s) \geq a_{k, \lambda} - \frac{1}{4} \delta.
\]

There also exists a $t(s, \delta)$ so that
\[
\left(\frac{t}{s}\right)^d \frac{1}{M^l(s, t)} \leq
1 + \frac{1}{4} \delta
\]
for all $t \geq t(s, \delta)$.

Noting that $\overline{F}_{k, \lambda}(s) \leq 1$
and $M^l(s, t) \leq \left(\frac{t}{s}\right)^d$, we then have
\begin{eqnarray*}
\left(\frac{t}{s}\right)^d \frac{1}{M^l(s, t)} (a_{k, \lambda} - \delta) &\leq&
\left(\frac{t}{s}\right)^d \frac{1}{M^l(s, t)}
\left(\overline{F}_{k, \lambda}(s) - \frac{3}{4} \delta\right)  \\
&\leq& 
\left(1 + \frac{1}{4} \delta\right) \left( \overline{F}_{k, \lambda}(s) -
\frac{3}{4} \delta \right) \\
&\leq& 
\overline{F}_{k, \lambda}(s) - \frac{1}{2} \delta \\
&=&
\expect{\widetilde{F}_{k, \lambda}(s)} - \frac{1}{2} \delta
\end{eqnarray*}
for all $t \geq t(s, \delta)$.

Using the stochastic lower bound~(\ref{lb1}), the above inequality
and~(\ref{chebyshev1}), we derive for all $t \geq t(s, \delta)$,
\begin{eqnarray*}
\pr{\widetilde{F}_{k, \lambda}(t) \leq a_{k, \lambda} - \delta}
&\leq&
\pr{\left(\frac{s}{t}\right)^d \sum_{i = 1}^{M_l(s, t)} \widetilde{N}_i(s) \leq
a_{k, \lambda} - \delta} \\
&=&
\pr{\frac{1}{M^l(s, t)} \sum_{i = 1}^{M^l(s, t)} \widetilde{N}_i(s) \leq
\left(\frac{t}{s}\right)^d \frac{1}{M^l(s, t)} (a_{k, \lambda} - \delta)} \\
&\leq&
\pr{\frac{1}{M^l(s, t)} \sum_{i = 1}^{M^l(s, t)} \widetilde{N}_i(s) \leq
\expect{\widetilde{F}_{k, \lambda}(s)} - \frac{1}{2} \delta} \\
&\leq&
\frac{4}{\lambda s^d M^l(s, t) \delta^2}.
\end{eqnarray*}

The latter term tends to~0 as $t$ grows large for any $\delta > 0$.

We now turn to the upper bound.

Proposition~\ref{prop:one} implies that there exists an $s = s_\delta$ such that
\[
\overline{F}_{k, \lambda}(s) \leq a_{k, \lambda} + \frac{1}{4} \delta.
\]

There also exists a $t(s, \delta)$ so that
\[
\left(\frac{t}{s}\right)^d \frac{1}{M^u(s, t)} \geq
1 - \frac{1}{8} \delta
\]
for all $t \geq t(s, \delta)$.

Noting that $\overline{F}_{k, \lambda}(t) \leq 1$ and assuming
$\delta \leq 1$, we then have
\[
\left(\frac{t}{s}\right)^d \frac{1}{M^u(s, t)} (a_{k, \lambda} + \delta) \geq
\left(\frac{t}{s}\right)^d \frac{1}{M^u(s, t)}
(\overline{F}_{k, \lambda}(s) + \frac{3}{4} \delta) \geq
\left(1 - \frac{1}{8} \delta\right)
(\overline{F}_{k, \lambda}(s) + \frac{3}{4} \delta) \geq
\]
\[
\overline{F}_{k, \lambda}(s) + \frac{5}{8} \delta - \frac{3}{32} \delta^2 \geq
\overline{F}_{k, \lambda}(s) + \frac{1}{2} \delta =
\expect{\widetilde{F}_{k, \lambda}(s)} + \frac{1}{2} \delta
\]
for all $t \geq t(s, \delta)$.

Using the stochastic upper bound~(\ref{ub2}), the above inequality
and~(\ref{chebyshev1}), we derive for all $t \geq t(s, \delta)$,
\begin{eqnarray*}
\pr{\widetilde{F}_{k, \lambda}(t) \geq a_{k, \lambda} + \delta}
&\leq&
\pr{\left(\frac{s}{t}\right)^d \sum_{i = 1}^{M_u(s, t)} \widetilde{N}_i(s) \geq
a_{k, \lambda} + \delta} \\
&=&
\pr{\frac{1}{M^u(s, t)} \sum_{i = 1}^{M^u(s, t)} \widetilde{N}_i(s) \geq
\left(\frac{t}{s}\right)^d \frac{1}{M^u(s, t)} (a_{k, \lambda} + \delta)} \\
&\leq&
\pr{\frac{1}{M^u(s, t)} \sum_{i = 1}^{M^u(s, t)} \widetilde{N}_i(s) \geq
\expect{\widetilde{F}_{k, \lambda}(s)} + \frac{1}{2} \delta} \\
&\leq&
\frac{4}{\lambda s^d M^u(s, t) \delta^2}.
\end{eqnarray*}

The latter term tends to~0 as $t$ grows large for any $\delta > 0$.
\end{proof}

\section{Evaluation of $a_{k,\lm}$}

In Propositions~\ref{prop:H.conv.in.prob} and~\ref{prop:lln1} we showed
that the fraction of nodes that can be properly colored converges
in probability to a constant $a_{k,\lm}$ as the size of the area
and the total number of nodes grow large.
It appears difficult to obtain an explicit expression for
$a_{k,\lm}$ in general.
Below we provide the exact value for the one-dimensional case $d = 1$
and present lower and upper bounds for $d > 1$.

\subsection{One-dimensional case (line)}

In case $d = 1$, a maximum proper coloring of points in the
interval $[0, t]$ can be obtained in a greedy manner by sequential
inspection of these points from left to right.
Specifically, a point at position $x \in [0, t]$ is selected
if fewer than $k$ points in $[x - 1, x]$ have been included,
and then assigned any color that is not already used for any points
in $[x - 1, x]$, and skipped otherwise.
If we now interpret the points in $[0, t]$ as arrival times of
customers, then it can be verified that the maximum proper coloring
obtained in the above fashion corresponds exactly to the arrival
times of those customers that are admitted in a so-called Erlang
loss system with service times~$1$ and capacity~$k$,
starting from an empty state at time~$0$.
In particular, the value of $F_{k,\lm}(t)$ equals the number
of admitted customers $A(k, \lm, t)$ in the latter Erlang loss system,
where the arrivals are governed by a Poisson process.
It is well known~\cite{kelly-1979-reversibility} that $A(k, \lm, t) / (\lambda t)$
converges
in probability to $1 - \mbox{Erl}(\lm, k)$
as $t \to \infty$, where
\[
\mbox{Erl}(\lm, k) =
\frac{\frac{\lm^k}{k!}}{\sum_{l = 0}^{k} \frac{\lm^l}{l!}} =
\frac{\pr{\poiss(\lm) = k}}{\pr{\poiss(\lm) \leq k}}
\]
denotes the so-called Erlang loss probability.
Thus
\[
a_{k,\lm} = 1 - \frac{\frac{\lm^k}{k!}}{\sum_{l = 0}^{k} \frac{\lm^l}{l!}} =
\frac{\sum_{m = 0}^{k - 1} \frac{\lm^m}{m!}}{\sum_{l = 0}^{k} \frac{\lm^l}{l!}} \,.
\]

\subsection{Higher dimensions}

For $d > 1$ it seems difficult to derive an explicit expression for
$a_{k,\lm}$.
In order to establish bounds, let $g_{\max}$ be the maximum density
that any set of points $V \subseteq \real_+^d$ can have such that
no two points are within distance~$1$.
Thus there is then a sphere of unit radius around every point in~$V$
that does not contain any other point in~$V$.
Hence a trivial upper bound is $g_{\max} \leq 1/V_d(1)$, where $V_d(1)$
is the volume of a sphere of unit radius in $d$~dimensions.
Now observe that any two points that are assigned the same color
cannot be within distance~$1$, i.e., the density of the points that
are assigned the same color can be at most $g_{\max}$.
This implies that $a_{k,\lm} \leq \max\{1, k g_{\max} / \lambda\}$.
While admittedly crude, we expect this upper bound to be asymptotically
tight for large values of~$\lm$, in the sense that
$\lambda a_{k,\lm} \to k g_{\max}$ as $\lambda \to \infty$.

For a given~$s$, let $m(s)$ be the minimum number of sets
$B_1, \dots, B_{m(s)}$ needed such that (i) any point in $\real^d$
is within distance~$s$ from one of the points
in $B_1 \cup \dots \cup B_{m(s)}$
and (ii) no two points within each of the sets $B_m$ are within
distance $1 + 2 s$, $m = 1, \dots, m(s)$.
The points in $B_1 \cup \dots \cup B_{m(s)}$ may be collectively
interpreted as `cell anchor points' or `base stations'.
Now suppose we partition the $k$~colors in $m(s)$ disjoint groups
$C_1, \dots, C_{m(s)}$ of size at least $\lfloor k / m(s) \rfloor$,
and then allocate all the colors in the group $C_m$ to all the
points in the set $B_m$, $m = 1, \dots, m(s)$.
In order to color a point, we will assign it one of the colors
allocated to the nearest anchor point.
Since any point in $\real^d$ is within distance~$s$ from one of the
anchor points, the points that need to be supported by a given anchor
point are all within radius~$s$, and hence the total number is bounded
from above by a Poisson random variable with parameter $\lm V_d(s)$.
This implies that $a_{k,\lm} \geq
\expect{\min\{\poiss(\lm V_d(s)), \lfloor k / m(s) \rfloor\}} / \lambda V_d(s)$.
While rather rough, this upper bound demonstrates that
$a_{k,\lm} \to 1$ as $k \to \infty$.
Thus, any target coloring ratio, however close to one, can be asymptotically
achieved with a sufficiently large but constant number of colors.

\section{Conclusion}

We have examined maximum vertex coloring of random geometric graphs,
in an arbitrary but fixed dimension.
This is a problem in discrete stochastic geometry which is neither
scale-invariant nor smooth, and hence the traditional methodological
framework to obtain limit laws cannot be applied.
We have therefore leveraged different concepts based on subadditivity
to establish convergence laws for the maximum number of vertices
that can be colored with a constant number of colors.
For the constants involved in these results, we derived the exact
value in dimension one, and upper and lower bounds in higher dimensions.

The approach that we have developed for this specific non-linear
Euclidean problem, offers great potential to be extended to a broader
class of problems, which will be the subject of further research. 
Some specific questions that we aim to address in future work are:
(1) What can one say when the `coverage area' is not a disk,
but some arbitrarily shaped area?
(2) Can one prove generalizations when the disk radius~$r$ is not
a constant, but a generally distributed random variable?

\bibliography{rggDAMreferences}

\begin{thebibliography}{10}

\bibitem{applegate:tsp:book}
David~L. Applegate, Robert~E. Bixby, Vasek Chvatal, and William~J. Cook.
\newblock {\em The Traveling Salesman Problem: A Computational Study (Princeton
  Series in Applied Mathematics)}.
\newblock Princeton University Press, Princeton, NJ, USA, 2007.

\bibitem{bhh-1959}
J.~Beardwood, J.~H. Halton, and J.~M. Hammersley.
\newblock The shortest path through many points.
\newblock {\em Proc. Cambridge Philos. Soc.}, 55:299--327, 1959.

\bibitem{brent-1990-unitdiskgraphs}
Brent~N. Clark, Charles~J. Colbourn, and David~S. Johnson.
\newblock {Unit disk graphs}.
\newblock {\em Discrete Mathematics}, 86(1-3):165--177, December 1990.

\bibitem{efron-1981-jackknife}
B.~Efron and C.~Stein.
\newblock The jackknife estimate of variance.
\newblock {\em Ann. Statist.}, 9(3):586--596, May 1981.

\bibitem{goemans-1991-prob}
Michel~X. Goemans and Dimitris~J. Bertsimas.
\newblock Probabilistic analysis of the {H}eld and {K}arp lower bound for the
  {E}uclidean traveling salesman problem.
\newblock {\em Mathematics of Operations Research}, 16(1):72--89, 1991.

\bibitem{kelly-1979-reversibility}
Frank~P. Kelly.
\newblock {\em Reversibility and stochastic networks}.
\newblock Wiley series in probability and mathematical statistics. Wiley, New
  York, Chichester, 1979.

\bibitem{DBLP:journals/im/LeskovecLDM09}
Jure Leskovec, Kevin~J. Lang, Anirban Dasgupta, and Michael~W. Mahoney.
\newblock Community structure in large networks: Natural cluster sizes and the
  absence of large well-defined clusters.
\newblock {\em Internet Mathematics}, 6(1):29--123, 2009.

\bibitem{mcdiarmid-2011-chromatic}
Colin McDiarmid and Tobias M\"uller.
\newblock On the chromatic number of random geometric graphs.
\newblock {\em Combinatorica}, 31(4):423--488, 2011.

\bibitem{molloy-reed-book-2002}
Michael Molloy and Bruce Reed.
\newblock {\em Graph Colouring and the Probabilistic Method}.
\newblock Springer-Verlag, 2002.

\bibitem{penrose:book}
Mathew~D. Penrose.
\newblock {\em Random Geometric Graphs}.
\newblock Oxford University Press, 2003.

\bibitem{rhee-1993-matching}
WanSoo~T. Rhee.
\newblock A matching problem and subadditive {E}uclidean functionals.
\newblock {\em Ann. Appl. Probab.}, 3(3):794--801, 1993.

\bibitem{subadditive-1981-steele}
J.~Michael Steele.
\newblock Subadditive euclidean functionals and nonlinear growth in geometric
  probability.
\newblock {\em Ann. Probab.}, 9(3):365--376, 1981.

\bibitem{steele-1997-probability}
J.~Michael Steele.
\newblock {\em Probability Theory and Combinatorial Optimization}.
\newblock CBMS-NSF Regional Conference Series in Applied Mathematics. Society
  for Industrial and Applied Mathematics, 1997.

\bibitem{yukich-2013-limit}
Joseph Yukich.
\newblock Limit theorems in discrete stochastic geometry.
\newblock In Evgeny Spodarev, editor, {\em Stochastic Geometry, Spatial
  Statistics and Random Fields}, volume 2068 of {\em Lecture Notes in
  Mathematics}, pages 239--275. Springer Berlin Heidelberg, 2013.

\end{thebibliography}
\bibliographystyle{plain}
\end{document}